\documentclass[12pt,twoside,reqno]{amsart}
\usepackage[italian,english]{babel}
\usepackage{amssymb,latexsym}
\usepackage{amsmath,amsfonts,amsthm}
\usepackage{paralist}
\usepackage{color}
\usepackage{amssymb, latexsym}

\usepackage{xspace}
\usepackage[bookmarksnumbered,colorlinks]{hyperref}
\usepackage{graphics}
\def\bb#1\eb{\textcolor{blue}
{#1}} %
\def\br#1\er{\textcolor{red}
{#1}} %
\def\bv#1\ev{\textcolor{green}
{#1}} %
\def\bc#1\ec{\textcolor{cyan}
{#1}} %

\usepackage{graphics}

\def\bal#1\eal{\begin{align}#1\end{align}}                      %
\def\baln#1\ealn{\begin{align*}#1\end{align*}}          
\def\bml#1\eml{\begin{multline}#1\end{multline}}        %
\def\bmln#1\emln{\begin{multline*}#1\end{multline*}}  %
\def\bga#1\ega{\begin{gather}#1\end{gather}}
\def\bgan#1\egan{\begin{gather*}#1\end{gather*}}

\makeatletter

\@addtoreset{equation}{section} \makeatother
\newtheorem{lem}{Lemma}
\newtheorem{prop}{Proposition}
\newtheorem{thm}{Theorem}

\newtheorem{defn}{Definition}
\theoremstyle{remark}
\newtheorem{remark}{Remark}

\newcommand{\s}{\sharp}

\DeclareMathOperator{\dive}{div}

\newcommand{\codim}{{\rm codim}}

\newcommand{\dimo}{\noindent{\em Proof.}$\ $}

\renewcommand{\>}{\rangle}

\newcommand{\RR}{\mathbb R}
\newcommand{\NN}{\mathbb N}

\def\R{{\mathbb R}}
\def\N{{\mathbb N}}
\def\ZZ{{\mathbb Z}}
\newcommand{\X}{\mathbb{X}^{s}_{T}}
\newcommand{\h}{\mathbb{H}^{s}_{T}}
\newcommand{\J}{\mathcal{J}}
\newcommand{\T}{\mathrm{Tr}}
\def\div{\mathop{\rm div}}

\renewcommand{\leq}{\leqslant}

\renewcommand{\geq}{\geqslant}

\newcommand{\Z}{\ensuremath{\mathbb{Z}}\xspace}     

\title[Weak periodic solutions]{Periodic solutions for a fractional asymptotically linear problem}

\author[V. Ambrosio]{Vincenzo Ambrosio}
\address{Vincenzo Ambrosio\hfill\break\indent
Dipartimento di Scienze Pure e Applicate (DiSPeA) \hfill\break\indent
Universit\`a degli Studi di Urbino `Carlo Bo' \hfill\break\indent
Piazza della Repubblica, 13\hfill\break\indent
61029 Urbino (Pesaro e Urbino) \hfill\break\indent
Italy}
\email{vincenzo.ambrosio@uniurb.it}

\author[G. Molica Bisci]{Giovanni Molica Bisci}
\address{Giovanni Molica Bisci\hfill\break\indent
Dipartimento PAU\hfill\break\indent
          Universit\`a `Mediterranea' di Reggio Calabria\hfill\break\indent
          Salita Melissari, Feo di
Vito, 89100 Reggio Calabria\hfill\break\indent
           Italy}
\email{gmolica@unirc.it}


\date{}

\begin{document}

\begin{abstract}
We study the existence and multiplicity of periodic weak solutions for a non-local equation involving an odd subcritical nonlinearity which is asymptotically linear at infinity. We investigate such problem by applying the the pseudo-index theory developed by Bartolo, Benci and Fortunato \cite{bbf} after transforming the problem to a degenerate elliptic problem in a half-cylinder with a Neumann boundary condition, via a Caffarelli-Silvestre type extension in periodic setting. The periodic nonlocal case, considered here, presents, respect to the cases studied in literature, some new additional difficulties and a
careful analysis of the fractional spaces involved is necessary.
\end{abstract}

\maketitle

\section{Introduction}
We consider here the non-local counterpart in periodic setting of the semilinear problem
\begin{equation}\label{DBP}
\left\{
\begin{array}{ll}
-\Delta u=g(x, u) &\mbox{ in } \Omega  \\
u=0    &\mbox{ on } \partial \Omega
\end{array},
\right.
\end{equation}
where $\Omega \subset \R^{N}$ is a smooth bounded domain and $g: \Omega \times \R \rightarrow \R$ is a given function asymptotically linear and possibly odd.\par
We notice that problem (\ref{DBP}) has been widely studied in the past years by many authors by using topological and variational methods (see, among others, \cite{az, bbf, r} and references therein).\par
In recent years, great attention has been devoted to the study of elliptic problems involving non-local operators which arise in a quite natural way in several areas of research; for more details and applications we refer to \cite{dpv} and the recent book \cite{MRS}.\par
However, although in literature there are many papers dealing with non-local fractional Laplacian equations with superlinear and sublinear growth \cite{A2, A3, A4, A5, AP0, AP, sv1, sv2}, only few papers consider asymptotically linear problems in fractional setting, see, for instance, \cite{bmb1, bmb2,bmr, fsv,MRSForum}.\par
\indent The aim of the present paper is to give a further result in this direction, considering a non-local problem with periodic boundary conditions.

More precisely, we are interested in the existence and multiplicity of the following problem
\begin{equation}\label{P}
\left\{
\begin{array}{ll}
(-\Delta+m^{2})^{s}u=\lambda_{\infty} u+ f(x,u) &\mbox{ in } (0,T)^{N}   \\
u(x+Te_{i})=u(x)    &\,\forall x \in \R^{N},\, i\in \ZZ[1,N]
\end{array}
\right.
\end{equation}
where $s\in (0,1)$, $N \geq 2$, $m>0$, $\lambda_{\infty}\in \R$ , $\{e_{i}\}_{i=1}^{N}$ is the canonical basis in $\R^{N}$, $f:\R^N\times \R\rightarrow \R$ is a Caratheodory function satisfying suitable assumptions, and $\ZZ[1,N]:=\{1,...,N\}$.\par
Moreover, the nonlocal operator $(-\Delta+m^{2})^{s}$ appearing in (\ref{P}) is defined by setting
\begin{equation*}\label{nfrls}
(-\Delta+m^{2})^{s} \,u(x):=\sum_{k\in \Z^{N}} \beta_{k} (\omega^{2}|k|^{2}+m^{2})^{s} \, \frac{e^{\imath \omega k\cdot x}}{{\sqrt{T^{N}}}},
\end{equation*}
for any $\displaystyle{u=\sum_{k\in \Z^{N}} \beta_{k} \frac{e^{\imath \omega k\cdot x}}{{\sqrt{T^{N}}}} \in \mathcal{C}^{\infty}_{T}(\R^{N})}$, where
$$
 \omega:=\frac{2\pi}{T}\mbox{ and } \; \beta_{k}:=\frac{1}{\sqrt{T^{N}}} \int_{(0,T)^{N}} u(x)e^{- \imath \omega k \cdot x}dx \quad (k\in \Z^{N})
$$
are the Fourier coefficients of the smooth and $T$-periodic function $u$.\par
\indent
This operator can be extended by density for every function that lies in the Hilbert space
$$
\mathbb{H}^{s}_{T}:=\left\{u=\sum_{k\in \Z^{N}} \beta_{k} \frac{e^{\imath \omega k\cdot x}}{{\sqrt{T^{N}}}}\in L^{2}(0,T)^{N}: \sum_{k\in \Z^{N}} (\omega^{2}|k|^{2}+m^{2})^{s} \, |\beta_{k}|^{2}<+\infty \right\}
$$
endowed by the norm
$$
|u|_{\mathbb{H}^{s}_{T}}:=\left(\sum_{k\in \Z^{N}} (\omega^{2}|k|^{2}+m^{2})^{s} |\beta_{k}|^{2}\right)^{1/2}.
$$

We also recall the embedding
properties of~$\mathbb{H}^{s}_{T}$ into the usual Lebesgue spaces; see Theorems \ref{tracethm} and \ref{compacttracethm} in
Subsection \ref{Sub1}. The embedding $j:\mathbb{H}^{s}_{T}\hookrightarrow
L^{\nu}(0,T)^{N}$ is continuous for any $\nu\in [1,2^{\sharp}_s]$, while it is
compact whenever $\nu\in [1,2^{\sharp}_s)$, where $2^{\sharp}_s:=2N/(N-2s)$ denotes the \textit{fractional critical Sobolev exponent}.

\indent
From a physical point of view, the meaning of the fractional operator $(-\Delta+m^{2})^{s}$ is manifest in the case $s=\frac{1}{2}$; in such case, $(-\Delta+m^{2})^{{1}/{2}}-m$ is the so-called free Hamiltonian which plays a fundamental role in relativistic quantum mechanic; see \cite{LL}.\par
Furthermore, $(-\Delta+m^{2})^{s}-m^{2s}$ is also related to the L\'{e}vy Processes theory; the operator in question is an infinitesimal generator of the relativistic $2s$-stable process  $\{X^{m}_{t}\}_{t \geq 0}$ with characteristic function given by
\begin{equation*}\label{carfun}
\mathbb{E}(e^{i \xi \cdot X^{m}_{t}}) := e^{-t [(m^{2}+|\xi|^{2})^{s}-m^{2s}]}  \quad (\xi \in \R^{N});
\end{equation*}
we refer to \cite{CMS} and \cite{ryznar}.

\smallskip
From now on we will suppose that $f:\R^{N+1}\rightarrow \R$ is a Caratheodory function satisfying the following assumptions:
\smallskip
\begin{compactenum}[($f_1$)]
\item $f(x,t)$ \textit{ is } $T$-\textit{periodic in} $x \in \R^{N}$, \textit{that is} $f(x+Te_{i},t)=f(x,t)$ \textit{for any} $x\in \R^{N}$ ($T\in \R$), \textit{for every} $i\in\ZZ[1,N]$ \textit{and}
$$
\sup_{|t|\leq a}|f(\cdot,t)|\in L^{\infty}(0, T)^{N} \mbox{ for any } a>0;
$$
\item \textit{there exist}
\begin{equation}\label{conmu2}
\lim_{|t|\rightarrow \infty} \frac{f(x, t)}{t}=0
\end{equation}
and
\begin{equation}
\lim_{t\rightarrow 0} \frac{f(x, t)}{t}=\lambda_{0}\in \R
\end{equation}
\textit{ uniformly with respect to a.e.} $x\in \R^{N}$.
\end{compactenum}
\smallskip

\indent
Let us denote respectively by $\sigma((-\Delta+m^{2})^{s})$ and by
$$0<\lambda_1<\lambda_2\leq\ldots\leq \lambda_k\leq\ldots$$ the spectrum and the non--decreasing, diverging sequence of the following problem
\begin{equation*}\label{Pauto}
\left\{
\begin{array}{ll}
(-\Delta+m^{2})^{s}u=\lambda u &\mbox{ in } (0,T)^{N}   \\
u(x+Te_{i})=u(x)    &\,\forall x \in \R^{N},\, i\in \ZZ[1,N]
\end{array}
\right.
\end{equation*}
\noindent repeated according to their multiplicity; by using the classical spectral theory \cite{RN, Y}, in Subsection \ref{spettro} we recall some features, which are very closed to the well known ones of $-\Delta$, about the spectrum of the operator $(-\Delta+m^{2})^{s}$.\par
\smallskip
With the above notation, our first result can be stated as follows
\begin{thm}\label{thm1}
Let $m>0,$ $s\in (0, 1)$ and $N\geq 2$. Suppose that $f:\R^{N+1} \rightarrow \R$ is a continuous function satisfying $(f_1)$ and $(f_2)$. Then, problem \eqref{P} has at least a weak solution in $\mathbb{H}^{s}_{T}$, provided that $\lambda_{\infty}\notin \sigma((-\Delta+m^{2})^{s})$.
\end{thm}

The main difficulty in our approach is essentially based on the fact that problem \eqref{P} is of non-nonlocal type. In addition, it doesn't seem to satisfy boundary conditions which would allow a direct application of variational methods without a certain extension procedure (see Section \ref{Sec2}). Following the work of Caffarelli and Silvestre \cite{C1}, the fractional Laplacian operator in the whole space $\RR^n$ can be defined
 as a Dirichlet to a Neumann map:
$$
(-\Delta)^{s}u(x):=-\kappa_s\lim_{y\rightarrow 0^+}y^{1-2s}\frac{\partial w}{\partial y}(x,y),
$$
where $\kappa_s$ is a suitable constant and $w$ is the $s$-harmonic extension of a smooth function $u$. In other words, $w$ is the function defined on the upper half-space $\RR^{n+1}_+:=\RR^n\times (0,+\infty)$ which is solution to the local elliptic problem
$$
 \left\{
\begin{array}{ll}
\displaystyle{ -\div(y^{1-2s}\nabla w)=  0 } \quad &\mbox{\rm in } \RR^{n+1}_+ \\
\,\,\,\,\, w(x,0)=u(x) \quad &\mbox{\rm in } \RR^{n}.
\end{array}\right.
$$

\indent In order to define the fractional Laplacian operator in bounded domains, the above procedure has been adapted in \cite{colorado0, cabretan}. To overcome the nonlocalicity of the operator $(-\Delta+m^{2})^{s}$ in (\ref{P}), we use the type extension in periodic setting developed in \cite{A2, A3}.\par
More precisely, for any $u\in \h$, there exists a unique $v\in \X$ which solves
\begin{equation*}
\left\{
\begin{array}{ll}
-\dive(y^{1-2s} \nabla v)+m^{2}y^{1-2s}v =0 &\mbox{ in }\mathcal{S}_{T}:=(0,T)^{N} \times (0,\infty)  \\
v_{| {\{x_{i}=0\}}}= v_{| {\{x_{i}=T\}}} & \mbox{ on } \partial_{L}\mathcal{S}_{T}:=\partial (0,T)^{N} \times [0,\infty) \\
v(x,0)=u(x)  &\mbox{ on }\partial^{0}\mathcal{S}_{T}:=(0,T)^{N} \times \{0\}
\end{array},
\right.
\end{equation*}
where $\X$ is the closure of the set of smooth and $T$-periodic (in $x$) functions in $\R^{N+1}_{+}$ with respect to the norm
$$
\|v\|_{\mathbb{X}_{T}^{s}}:=\left(\iint_{\mathcal{S}_{T}} y^{1-2s} (|\nabla v|^{2}+m^{2s} v^{2}) dx dy\right)^{1/2}.
$$
\indent Then, the operator $(-\Delta+m^{2})^{s}$ is obtained as
$$
-\lim_{y\rightarrow 0^{+}} y^{1-2s} \frac{\partial v}{\partial y}(x,y) = \kappa_{s} (-\Delta + m^{2})^{s} u(x)
$$
in weak sense and $\displaystyle{\kappa_{s}:= 2^{1-2s} \frac{\Gamma(1-s)}{\Gamma(s)}}$.

Therefore, we exploit this fact, and instead of (\ref{P}), we investigate the following problem
\begin{equation}\label{R}
\left\{
\begin{array}{ll}
-\dive(y^{1-2s} \nabla v)+m^{2}y^{1-2s}v =0 &\mbox{ in }\mathcal{S}_{T}  \\
\smallskip
v_{| {\{x_{i}=0\}}}= v_{| {\{x_{i}=T\}}} & \mbox{ on } \partial_{L}\mathcal{S}_{T} \\
\smallskip
{\partial_{\nu}^{1-2s} v}=\kappa_{s} [\lambda_{\infty} v+ f(x,v)]   &\mbox{ on }\partial^{0}\mathcal{S}_{T}
\end{array}
\right.
\end{equation}
where
$$
{\partial_{\nu}^{1-2s} v}(x):=-\lim_{y \rightarrow 0^{+}} y^{1-2s} \frac{\partial v}{\partial y}(x,y)
$$
is the conormal exterior derivative of $v$.\par
Since $(\ref{R})$ has a variational structure, its solutions can be found as critical points of the energy functional $\mathcal{J}$ given by
 $$
 \mathcal{J}(v):=\frac{1}{2}\left( \|v\|_{\mathbb{X}_{T}^{s}}^{2}-\lambda_{\infty} \kappa_{s} |\textup{Tr}(v)|_{L^{2}(0,T)^{N}}^{2}\right) -\kappa_{s}\int_{\partial^{0}\mathcal{S}_{T}} F(x,\textup{Tr}(v)) \,dx,
 $$
defined on the space $\mathbb{X}_{T}^{s}$.\par

A natural question
is whether or not these topological and variational methods may be adapted to
equation \eqref{P} and to its generalization in order to extend the classical results known
for \eqref{DBP} to a non-local periodic context.\par

\indent In the first part of the paper we prove that the geometry of the classical Saddle Point Theorem
due to Rabinowitz \cite{r} is respected by the non-local framework: for this we use a functional
analytical setting that is inspired by (but not equivalent to) the fractional Sobolev
spaces, in order to correctly encode the periodic boundary datum in the variational
formulation. Of course, also the compactness property required by this abstract theorem is satisfied
in the non-local setting, again thanks to the choice of the functional setting we work in.\par

\indent In addition, when the nonlinear-term is symmetric, we obtain the existence of multiple periodic solutions to (\ref{P}). More precisely, we adapt the pseudo-index theory due to Bartolo, Benci and Fortunato \cite{bbf}, to prove a multiplicity result for critical points of the even energy functional $\J$; see also \cite{[BCS2],[BCS], bmb1,bmb2} for related topics.\par
\smallskip
More precisely, we are able to prove the following result
\begin{thm}\label{thm2}
Let $m>0,$ $s\in (0, 1)$ and $N\geq 2$. Suppose that $f:\R^{N+1} \rightarrow \R$ is a continuous function satisfying conditions $(f_1)$ and $(f_2)$. Further, let us assume that $f(x, \cdot)$ is odd for a.e. $x\in \R^{N}$ and require that\par
\begin{itemize}
\item[$(C_{\lambda_{\infty}})$] there exist $h, k\in \N$, with $k \geq h$, such that
$$
\lambda_{0}+\lambda_{\infty}<\lambda_{h}\leq \lambda_{k}<\lambda_{\infty}.
$$
\end{itemize}
Then, problem $\eqref{P}$ has at least $k-h+1$ distinct pairs of non-trivial weak solutions in $\mathbb{H}^{s}_{T},$ provided that $\lambda_{\infty}\notin \sigma((-\Delta+m^{2})^{s})$.
\end{thm}

We emphasize that further difficulties arise in the so--called  ``resonant case'', that is $\lambda_{\infty}\in \sigma((-\Delta+m^{2})^{s})$: indeed, the resonance affects
both the compactness property and
the geometry of the Euler--Lagrange functional arising in a suitable variational approach (cf., e.g., \cite{bbf} and references therein for the classical elliptic case). We will consider this interesting case via some further investigations. Finally, we also point out that our results should be viewed as a periodic nonlocal version of Proposition 1.1 and Theorem 1.2 of \cite{bmb2}.\par

For the sake of completeness we point out that fractional nonlocal equations have by now been widely investigated
from many points of view. Besides \cite{AP0,AP,cabretan,pu1,pu2, secchio1,secchio2} we cite \cite{kms1,kms2,PP,PPNA} and the very recent
paper \cite{va} and references therein for some nice interpretations of nonlocal physical phenomena.\par

As far as we know the results presented here are new for periodic fractional problems. The body of the paper is as follows: in Section $2$ we collect some preliminaries related to the functional setting and the abstract critical point theory that we use to study (\ref{R}): in Section $3$ we prove the existence of a weak solution to (\ref{P}), while in the last Section $4$ we give the proof of Theorem \ref{thm2}.
\section{Preliminaries} \label{Sec2}
\subsection{Functional setting}\label{Sub1}
This section is devoted to the notations used along the present paper. In order to give the weak formulation of problem~\eqref{P}, we need to work in a special functional space. Indeed, one of the difficulties in treating problem~\eqref{P} is related to encoding the periodic boundary condition in the variational formulation. With this respect the standard fractional Sobolev spaces are not sufficient in order to study the problem. We overcome this difficulty by working in a new functional space, whose definition is  recalled here. From now on, we assume $s\in (0, 1)$ and $N\geq 2$. \par
\indent Let
$$
\R^{N+1}_{+}=\{(x,y)\in \R^{N+1}: x\in \R^{N}, y>0 \}
$$
be the upper half-space in $\R^{N+1}$.\par
\indent Let $\mathcal{S}_{T}:=(0,T)^{N}\times(0,\infty)$ be the half-cylinder in $\R^{N+1}_{+}$ with basis $\partial^{0}\mathcal{S}_{T}:=(0,T)^{N}\times \{0\}$ and we denote by $\partial_{L}\mathcal{S}_{T}:=\partial (0,T)^{N}\times [0,+\infty)$ the lateral boundary of $\mathcal{S}_{T}$. \par
With $\|v\|_{L^{r}(\mathcal{S}_{T})}$ we will always denote the norm of $v\in L^{r}(\mathcal{S}_{T})$ and with $|u|_{L^{r}(0,T)^{N}}$ the $L^{r}(0,T)^{N}$ norm of $u \in L^{r}(0,T)^{N}$.\par
Let $\mathcal{C}^{\infty}_{T}(\R^{N})$ be the space of functions
$u\in \mathcal{C}^{\infty}(\R^{N})$ such that $u$ is $T$-periodic in each variable, that is
$$
u(x+Te_{i})=u(x) \mbox{ for all } x\in \R^{N}, i\in \ZZ[1,N].
$$
\indent We define the fractional Sobolev space $\mathbb{H}^{s}_{T}$ as the closure of $\mathcal{C}^{\infty}_{T}(\R^{N})$ endowed by the norm
\begin{equation*}\label{h12norm}
|u|_{\mathbb{H}^{s}_{T}}:=\sqrt{ \sum_{k\in \Z^{N}} (\omega^{2}|k|^{2}+m^{2})^{s} \, |\beta_{k}|^{2}}.
\end{equation*}
\indent Let us introduce the functional space $\mathbb{X}^{s}_{T}$ defined as the completion of
\begin{align*}
\mathcal{C}_{T}^{\infty}(\overline{\R^{N+1}_{+}}):=\Bigl\{&v\in \mathcal{C}^{\infty}(\overline{\R^{N+1}_{+}}): v(x+Te_{i},y)=v(x,y) \\
&\mbox{ for every } (x,y)\in \overline{\R_{+}^{N+1}}, i=1, \dots, N \Bigr\}
\end{align*}
under the $H^{1}(\mathcal{S}_{T},y^{1-2s})$ norm
\begin{equation*}
\|v\|_{\mathbb{X}^{s}_{T}}:=\sqrt{\iint_{\mathcal{S}_{T}} y^{1-2s} (|\nabla v|^{2}+m^{2}v^{2}) \, dxdy} \,.
\end{equation*}
\indent We recall that it is possible to define a trace operator between $\mathbb{X}^{s}_{T}$ and $\mathbb{H}_{T}^{s}$ (see \cite{A2, A3} for details):
\begin{thm}\label{tracethm}
There exists a surjective linear operator $\textup{Tr} : \mathbb{X}^{s}_{T} \rightarrow \mathbb{H}_{T}^{s}$  such that:
\begin{itemize}
\item[$(i)$] $\textup{Tr}(v)=v|_{\partial^{0} \mathcal{S}_{T}}$ for all $v\in \mathcal{C}_{T}^{\infty}(\overline{\R^{N+1}_{+}}) \cap \mathbb{X}^{s}_{T}$;
\item[$(ii)$] $\textup{Tr}$ is bounded and
\begin{equation}\label{tracein}
\sqrt{\kappa_{s}} |\textup{Tr}(v)|_{\mathbb{H}^{s}_{T}}\leq \|v\|_{\mathbb{X}^{s}_{T}},
\end{equation}
for every $v\in \mathbb{X}^{s}_{T}$.
In particular, equality holds in \eqref{tracein} for some $v\in \mathbb{X}^{s}_{T}$ if and only if $v$ weakly solves the following equation
$$
-\dive(y^{1-2s} \nabla v)+m^{2}y^{1-2s}v =0 \, \mbox{ in } \, \mathcal{S}_{T}.
$$
\end{itemize}
\end{thm}

Moreover, we have the following embedding results:
\begin{thm}\label{compacttracethm}
Let $N\geq 2$ and $s\in (0,1)$. Then  $\textup{Tr}(\mathbb{X}^{s}_{T})$ is continuously embedded in $L^{q}(0,T)^{N}$ for any  $1\leq q \leq 2^{\s}_{s}$.  Moreover,  $\textup{Tr}(\mathbb{X}^{s}_{T})$ is compactly embedded in $L^{q}(0,T)^{N}$  for any  $1\leq q < 2^{\s}_{s}$.
\end{thm}

Now, we aim to reformulate the nonlocal problem (\ref{P}) in a local way.
Let $g \in \mathbb{H}^{-s}_{T}$, where
$$
\mathbb{H}^{-s}_{T}:=\left\{g=\sum_{k\in \Z^{N}} g_{k} \frac{e^{\imath \omega k\cdot x}}{{\sqrt{T^{N}}}}: \sum_{k\in \Z^{N}}\frac{|g_{k}|^{2}}{(\omega^{2}|k|^{2}+m^{2})^{s}}<+\infty\right\}
$$
is the dual of $\mathbb{H}^{s}_{T}$, and consider the following two problems:
\begin{equation}\label{P*}
\left\{
\begin{array}{ll}
(-\Delta+m^{2})^{s}u=g &\mbox{ in } (0,T)^{N}  \\
u(x+Te_{i})=u(x) & \mbox{ for all } x\in \R^{N}
\end{array}
\right.
\end{equation}
and
\begin{equation}\label{P**}
\left\{
\begin{array}{ll}
-\dive(y^{1-2s} \nabla v)+m^{2}y^{1-2s}v =0 &\mbox{ in }\mathcal{S}_{T}  \\
v_{| {\{x_{i}=0\}}}= v_{| {\{x_{i}=T\}}} & \mbox{ on } \partial_{L}\mathcal{S}_{T} \\
{\partial_{\nu}^{1-2s} v}=\kappa_{s} g(x)  &\mbox{ on } \partial^{0}\mathcal{S}_{T}
\end{array}.
\right.
\end{equation}
\indent Then we have the following definitions of weak solutions to \eqref{P**} and \eqref{P*} respectively:
\begin{defn}
We say that $v \in \mathbb{X}^{s}_{T}$ is a weak solution to \eqref{P**}
if for every $\varphi \in \mathbb{X}^{s}_{T}$ it holds
$$
\iint_{\mathcal{S}_{T}} y^{1-2s} (\nabla v \nabla \varphi + m^{2} v \varphi  ) \, dxdy=\kappa_{s}\langle g, \textup{Tr}(\varphi)\rangle.
$$
Here $\langle \cdot, \cdot \rangle$ is the duality pairing between $\mathbb{H}^{s}_{T}$ and its dual $\mathbb{H}^{-s}_{T}$.
\end{defn}
\begin{defn}
We say that $u\in \mathbb{H}^{s}_{T}$ is a weak solution to \eqref{P*} if $u=\textup{Tr}(v)$
and $v\in \mathbb{X}^{s}_{T}$ is a weak solution to \eqref{P**}.
\end{defn}

Taking into account Theorem \ref{tracethm} and Theorem \ref{compacttracethm}, it is possible to introduce the notion of extension for a function $u\in \mathbb{H}^{s}_{T}$.\par
 More precisely, the next result holds.
\begin{thm}
Let $u\in \mathbb{H}^{s}_{T}$. Then, there exists a unique $v\in \mathbb{X}^{s}_{T}$ such that
\begin{equation}\label{extPu}
\left\{
\begin{array}{ll}
-\dive(y^{1-2s} \nabla v)+m^{2}y^{1-2s}v =0 &\mbox{ in }\mathcal{S}_{T}  \\
v_{| {\{x_{i}=0\}}}= v_{| {\{x_{i}=T\}}} & \mbox{ on } \partial_{L}\mathcal{S}_{T} \\
v(\cdot,0)=u  &\mbox{ on } \partial^{0}\mathcal{S}_{T}
\end{array}
\right.
\end{equation}
and
\begin{align*}\label{conormal}
-\lim_{y \rightarrow 0^{+}} y^{1-2s}\frac{\partial v}{\partial y}(x,y)=\kappa_{s} (-\Delta+m^{2})^{s}u(x) \mbox{ in } \mathbb{H}^{-s}_{T}.
\end{align*}
We call $v\in \mathbb{X}^{s}_{T}$ the extension of $u\in \mathbb{H}^{s}_{T}$.\par
In particular, if $u=\displaystyle\sum_{k\in \Z^{N}} \beta_{k} \frac{e^{\imath \omega k\cdot x}}{{\sqrt{T^{N}}}}$, then $v$ is given by
\begin{align*}
v(x,y)=\sum_{k\in \Z^{N}} \beta_{k} \theta_{k}(y) \frac{e^{\imath \omega k\cdot x}}{{\sqrt{T^{N}}}},
\end{align*}
where $\theta_{k}(y):= \theta(\sqrt{\omega^{2} |k|^{2}+ m^{2}} y)$ and $\theta(y)\in H^{1}(\R_{+},y^{1-2s})$ solves the following ODE
\begin{equation*} \label{ccv}
\left\{
\begin{array}{cc}
&\theta{''}+\displaystyle\frac{1-2s}{y}\theta{'}-\theta=0 \mbox{ in } \R_{+}  \\
&\theta(0)=1 \mbox{ and } \theta(\infty)=0
\end{array}.
\right.
\end{equation*}
\indent Moreover, $v$ satisfies the properties:
\begin{itemize}
\item[$(i)$] $v$ is smooth for $y>0$ and $T$-periodic in $x$$;$
\item[$(ii)$] $\|v\|_{\mathbb{X}^{s}_{T}}\leq \|z\|_{\mathbb{X}^{s}_{T}}$ for any $z\in \mathbb{X}^{s}_{T}$ such that $\textup{Tr}(z)=u$$;$
\item[$(iii)$] $\|v\|_{\mathbb{X}^{s}_{T}}=\sqrt{\kappa_{s}} |u|_{\mathbb{H}^{s}_{T}}$.
\end{itemize}
\end{thm}

\indent
Now, we are in the position to reformulate the nonlocal problem (\ref{P}) with periodic boundary conditions, in a local way according the following definitions.
\begin{defn}
We say that $v \in \mathbb{X}^{s}_{T}$ is a weak solution to \eqref{R}
if
$$
\iint_{\mathcal{S}_{T}} y^{1-2s} (\nabla v \nabla \varphi + m^{2} v \varphi  ) \, dxdy=\kappa_s\gamma_{\infty} \int_{\partial^{0}\mathcal{S}_{T}} \textup{Tr}(v)\textup{Tr}(\varphi) dx
$$
$$
\qquad+\lambda\kappa_s\int_{\partial^{0}\mathcal{S}_{T}} f(x, \textup{Tr}(v)) \textup{Tr}(\varphi) dx,
$$
for every $\varphi \in \mathbb{X}^{s}_{T}$.
\end{defn}
Finally, we give the notion of \textit{weak solution} to problem \eqref{P}.
\begin{defn}
We say that $u\in \mathbb{H}^{s}_{T}$ is a weak solution to \eqref{P} if $u=\textup{Tr}(v)$
and $v\in \mathbb{X}^{s}_{T}$ is a weak solution to \eqref{R}.
\end{defn}

\subsection{On periodic eigenvalue problem in the half-cylinder}\label{spettro}
The study of the eigenvalues of a linear operator is a classical topic and many
functional analytic tools of general flavor may be used to deal with it. The result
that we give here is, indeed, more general and more precise than what we need,
strictly speaking, for the proofs of our main results: nevertheless we believed it
was good to have a result stated in detail also for further
reference. Hence, we focus on the following eigenvalue problem
\begin{equation}\label{Pauto2}
\left\{
\begin{array}{ll}
(-\Delta+m^{2})^{s}u=\lambda u &\mbox{ in } (0,T)^{N}   \\
u(x+Te_{i})=u(x)    &\,\forall x \in \R^{N},\, i\in \ZZ[1,N]
\end{array}
\right..
\end{equation}

More precisely, we discuss the weak formulation of \eqref{Pauto2}, which consists in the
following eigenvalue problem:\par
Find $u\in \mathbb{H}^{s}_{T}$ such that $u=\textup{Tr}(v)$
where $v\in \mathbb{X}^{s}_{T}$ and
\begin{align*}\begin{split}
&\iint_{\mathcal{S}_{T}} y^{1-2s} (\nabla v \nabla \varphi + m^{2} v \varphi  ) \, dxdy \\
&\,\,\,\,\,\,\,\,\,\,\,\,\,\,\,\,\,\,\,\,\,\,\,=\lambda\kappa_s\int_{\partial^{0}\mathcal{S}_{T}} \textup{Tr}(v) \textup{Tr}(\varphi) dx,
\end{split}\end{align*}
for every $\varphi \in \mathbb{X}^{s}_{T}$.\par
 In other words, $v\in \mathbb{X}^{s}_{T}$ is a weak solution of the extended problem
\begin{equation}\label{Rer}
\left\{
\begin{array}{ll}
-\dive(y^{1-2s} \nabla v)+m^{2}y^{1-2s}v =0 &\mbox{ in }\mathcal{S}_{T}  \\
\smallskip
v_{| {\{x_{i}=0\}}}= v_{| {\{x_{i}=T\}}} & \mbox{ on } \partial_{L}\mathcal{S}_{T} \\
\smallskip
-\displaystyle\lim_{y \rightarrow 0^{+}} y^{1-2s} \frac{\partial v}{\partial y}(x,y)=\lambda\kappa_{s} v   &\mbox{ on }\partial^{0}\mathcal{S}_{T}
\end{array}
\right..
\end{equation}

We recall that $\lambda\in\RR$ is an eigenvalue of $(-\Delta+m^{2})^{s}$ provided there exists a non-trivial
weak solution of \eqref{Pauto2}.\par

\indent Following the classical spectral theory \cite{RN, Y}, the powers of a positive operator in a bounded domain are defined through the spectral decomposition using the powers of the eigenvalues of the original operator.\par
 \smallskip
Then, we can derive the following result:
\begin{lem}\label{lemmino}With the above notations the following facts hold:
\begin{itemize}
\item[$(i)$] The operator $(-\Delta+m^{2})^{s}$ has a countable family of eigenvalues $\{\lambda_{\ell}\}_{\ell\in \N}$ which can be written as an increasing sequence of positive numbers
$$
0<\lambda_{1}<\lambda_{2}\leq \dots\leq \lambda_{\ell}\leq \lambda_{\ell+1}\leq \dots
$$
Each eigenvalue is repeated a number of times equal to its multiplicity $\mathopen{(}$which is finite$\mathclose{)}$$;$
 \item[$(ii)$] $\lambda_{\ell}=\mu_{\ell}^{s}$ for all $\ell\in \N$, where $\{\mu_{\ell}\}_{\ell\in \N}$ is the increasing sequence of eigenvalues of $-\Delta+m^{2}$$;$
\item[$(iii)$] $\lambda_{1}=m^{2s}$ is simple, $\lambda_{\ell}=\mu_{\ell}^{s} \rightarrow +\infty$ as $\ell \rightarrow +\infty$$;$
 \item[$(iv)$] The sequence $\{u_{\ell}\}_{\ell\in \N}$ of eigenfunctions corresponding to $\lambda_{\ell}$ is an orthonormal basis of $L^{2}(0, T)^{N}$ and an orthogonal basis of the Sobolevspace $\mathbb{H}^{s}_{T}$.
Let us note that $\{u_{\ell}, \mu_{\ell}\}_{\ell\in \N}$ are the eigenfunctions and eigenvalues of $-\Delta+m^{2}$ under periodic boundary conditions$;$
\item[$(v)$] For any $h\in \N$, $\lambda_{\ell}$ has finite multiplicity, and there holds
$$
\lambda_{\ell}=\min_{u\in \mathbb{P}_{\ell}\setminus \{0\}} \frac{|u|^{2}_{\mathbb{H}^{s}_{T}}}{|u|^{2}_{L^{2}(0, T)^{N}}}   \quad \mbox{$\mathopen{(}$Rayleigh's principle$\mathclose{)}$}
$$
where
$$
\mathbb{P}_{\ell}:=\{u\in \mathbb{H}^{s}_{T}: \langle u, u_{j} \rangle_{\mathbb{H}^{s}_{T}} =0, \mbox{ for } j=1, \dots, \ell-1\}.
$$
\end{itemize}
\end{lem}
\begin{proof}
It is enough to prove that $(-\Delta+m^{2})^{-s}:L^{2}(0,T)^{N} \rightarrow L^{2}(0,T)^{N}$ is a self-adjoint, positive  and compact operator.\\
Firstly we observe that if $u=\sum_{k\in \Z^{N}} c_{k} e^{\imath \omega k\cdot x}$ and $v=\sum_{k\in \Z^{N}} d_{k} e^{\imath \omega k\cdot x}$ belong to $L^{2}(0,T)^{N}$, then
\begin{align*}
\langle(-\Delta+m^{2})^{-s}u, v\rangle_{L^{2}(0, T)^{N}}&=\sum_{k\in \Z^{N}} \frac{c_{k}}{(\omega^{2}|k|^{2}+m^{2})^{s}} \bar{d}_{k}\\
&=\sum_{k\in \Z^{N}} c_{k}\frac{\bar{d}_{k}}{(\omega^{2}|k|^{2}+m^{2})^{s}}\\
&=\langle u, (-\Delta+m^{2})^{-s}v\rangle_{L^{2}(0, T)^{N}}
\end{align*}
that is $(-\Delta+m^{2})^{-s}$ is self-adjoint.\\
Clearly, for $u=\sum_{k\in \Z^{N}} c_{k} e^{\imath \omega k\cdot x}\in L^{2}(0,T)^{N}$ we have
$$
\langle(-\Delta+m^{2})^{-s}u, u\rangle_{L^{2}(0, T)^{N}}=\sum_{k\in \Z^{N}} \frac{|c_{k}|^{2}}{(\omega^{2}|k|^{2}+m^{2})^{s}}\geq 0
$$
and $\langle (-\Delta+m^{2})^{-s}u, u\rangle_{L^{2}(0, T)^{N}}>0$ if $u\neq 0$.

\indent
Finally, we show that $(-\Delta+m^{2})^{-s}$ is compact. Let $\{v_{j}\}_{j\in \N}$ be a bounded sequence in $L^{2}(0, T)^{N}$ and let us denote by $\{d^{j}_{k}\}_{k\in \Z^{N}}$ its Fourier coefficients.. Since $L^{2}(0, T)^{N}\subset \mathbb{H}_{T}^{-s}$, it is clear that $\{v_{j}\}_{j\in \N}$ is bounded in $\mathbb{H}_{T}^{-s}$.
Let us denote by $u_{j}=(-\Delta+m^{2})^{-s} v_{j}$. \\
Then 
$$ 
|u_{j}|^{2}_{\mathbb{H}^{s}_{T}}=\sum_{k\in \Z^{N}}(\omega^{2}|k|^{2}+m^{2})^{s} \frac{|d^{j}_{k}|^{2}}{(\omega^{2}|k|^{2}+m^{2})^{2s}}=|v_{j}|_{\mathbb{H}^{-s}_{T}}^{2},
$$
that is $\{u_{j}\}_{j\in \N}$ is a bounded sequence in $\mathbb{H}^{s}_{T}$. By using the compactness of $\mathbb{H}^{s}_{T}$ into $L^{2}(0, T)^{N}$, we deduce that $\{u_{j}\}_{j\in \N}$ admits a convergent subsequence in $L^{2}(0, T)^{N}$. As a consequence $(-\Delta+m^{2})^{-s} v_{j}$ strongly converges in $L^{2}(0, T)^{N}$.
\end{proof}

\indent
Now, we aim to find some useful relation between the eigenvalues $\{\lambda_{j}\}_{j\in\NN}$ of $(-\Delta+m^{2})^{s}$ and the corresponding extended eigenvalue problem in the half-cylinder $\mathcal{S}_{T}$,
\begin{equation}\label{ERP}
\left\{
\begin{array}{ll}
-\dive(y^{1-2s} \nabla v)+m^{2}y^{1-2s}v =0 &\mbox{ in }\mathcal{S}_{T} \\
\smallskip
v_{| {\{x_{i}=0\}}}= v_{| {\{x_{i}=T\}}} & \mbox{ on } \partial_{L}\mathcal{S}_{T}\\
\smallskip
{\partial_{\nu}^{1-2s} v}=\lambda_{j}\kappa_{s} v   &\mbox{ on }\partial^{0}\mathcal{S}_{T}
\end{array}.
\right.
\end{equation}
\indent Let us introduce the following notations. Set
\begin{equation}\label{defvh}
\mathbb{V}_{h}:=Span\{v_{1}, \dots, v_{h}\},
\end{equation}
where every $v_{j}$ solves (\ref{ERP}).
Clearly $\T(v_{j})=u_{j}$ for all $j\in \N$, where $\{u_{j}\}_{j\in\NN}$ is the basis of eigenfunctions in $\h$, defined in Lemma \ref{lemmino}. For any $h\in \N$ we define
\begin{equation}\label{defvhort}
\mathbb{V}_{h}^{\perp}:=\{v\in \X: \langle v, v_{j} \rangle_{\X} =0, \mbox{ for } j=1, \dots, h\}.
\end{equation}
\indent Since $v_{j}$ solves (\ref{ERP}), then we deduce that
$$
\mathbb{V}_{h}^{\perp}=\{v\in \X: \langle \T(v), \T(v_{j}) \rangle_{L^{2}(0, T)^{N}} =0, \mbox{ for } j=1, \dots, h\}.
$$
Then $ \X=\mathbb{V}_{h}\bigoplus \mathbb{V}_{h}^{\perp}$. Let us point out that the trace operator is bijective on $E:=\{v\in \X: v \mbox{ solves } (\ref{extPu})\}$.\par
 Indeed, if $\tilde{v}_{1}$ and $\tilde{v}_{2}$ are the extension of $\tilde{u}_{1},\tilde{u}_{2} \in \h$ respectively, then
\begin{equation}\label{*}
\langle \tilde{v}_{i}, \varphi \rangle_{\X} = k_{s} \langle \tilde{u}_{i}, \T(\varphi) \rangle_{\h} \quad \forall \varphi \in \X, i=1,2.
\end{equation}
If $\tilde{u}_{1}=\T(\tilde{v}_{1})= \T(\tilde{v}_{2})=\tilde{u}_{2}$, from (\ref{*}) follows that
$$
\langle \tilde{v}_{1}-\tilde{v}_{2}, \varphi \rangle_{\X}=0 \quad \forall \varphi \in \X,
$$
so we deduce that $\tilde{v}_{1}=\tilde{v}_{2}$, that is $\T$ is injective on $E$.\par
\indent From this and the linearity of the trace operator $\T$, we have
$$
\dim \mathbb{V}_{h}=\dim Span\{\T(v_{1}), \cdots, \T(v_{h})\}=h.
$$
Now we prove that $\|\cdot\|_{\X}$ and $|\cdot|_{2}$ are equivalent norms on the finite dimensional space $\mathbb{V}_{h}$.\par More precisely, for any $v\in \mathbb{V}_{h}$, it results
\begin{align}\label{eqnorm}
\kappa_{s}m^{2s} |\T(v)|_{2}^{2}\leq \|v\|_{\mathbb{X}^{s}}^{2}\leq \kappa_{s}\lambda_{h} |\T(v)|_{2}^{2}.
\end{align}
Firstly we note that $\{v_{j}\}_{j\in \N}$ is an orthogonal system in $\X$, since $\{\T(v_{j})\}_{j\in \N}$ is an orthonormal system in $L^{2}(0, T)^{N}$, and $v_{j}$ satisfies
$$
\langle z, v_{j} \rangle_{\X}=\kappa_{s} \lambda_{j}\langle \T(z), \T(v_{j}) \rangle_{L^{2}(0, T)^{N}} \mbox{ for all } z\in \mathbb{X}^{s}, j\in \N.
$$
\indent Then, by using the fact that $\{\lambda_{j}\}_{j\in \N}$ is an increasing sequence (see $(i)$ of Lemma \ref{lemmino}), and by trace inequality (\ref{tracein}), for $v=\sum_{j=1}^{h} \alpha_{j} v_{j}\in \mathbb{V}_{h}$ we have
\begin{align*}
\kappa_{s}m^{2s} |\T(v)|_{2}^{2}&\leq \|v\|_{\mathbb{X}^{s}}^{2}=\sum_{j=1}^{h} \alpha_{j}^{2}\|v_{j}\|^{2}_{\mathbb{X}^{s}}\\
&=\kappa_{s}\sum_{j=1}^{h} \lambda_{j} \alpha_{j}^{2}|\T(v_{j})|^{2}_{2}\leq \kappa_{s}\lambda_{h}\sum_{j=1}^{h} \alpha_{j}^{2}|\T(v_{j})|^{2}_{2}\\
&= \kappa_{s}\lambda_{h} |\T(v)|_{2}^{2}. 
\end{align*}
\indent Finally we prove that for any $v\in \mathbb{V}_{h}^{\perp}$ the following inequality holds
$$
\lambda_{h+1} |\T(v)|_{2}^{2}\leq \frac{1}{\kappa_{s}}\|v\|^{2}_{\X}.
$$
Fix $v\in \mathbb{V}_{h}^{\perp}$. Then $\T(v)\in \mathbb{P}_{h+1}$.
Indeed $\T(v_{j})=u_{j}$ is a weak solution to $(-\Delta+m^{2})^{s}u=\lambda_{j} u$ and by using the fact that
$$
\langle \T(v), \T(v_{j}) \rangle_{L^{2}(0, T)^{N}}=0 \mbox{ for every } j=1, \dots, h,
$$
we can infer that $\langle \T(v), \T(v_{j}) \rangle_{\h}=0,$ for every $j=1, \dots, h$.\par
As a consequence, by using the variational characterization $(v)$ of Lemma \ref{lemmino} and the trace inequality (\ref{tracein}), we get
\begin{equation}\label{C}
\lambda_{h+1} |\T(v)|_{2}^{2}\leq  |\T(v)|^{2}_{\h}\leq \frac{1}{\kappa_{s}}\|v\|^{2}_{\X}.
\end{equation}

\subsection{An abstract critical point theorem}\label{INDI}

\noindent
Since (\ref{R}) has a variational structure, in this section we provide the main classical tools that we use for its study.

Let us denote by $(E,\|\cdot\|_E)$ a Banach space,
$(E',\|\cdot\|_{E'})$ its dual, $\Phi$ a $C^1$ functional on $E$, $\Phi^b := \{e\in E : \Phi(e)\leq b\}$ the sublevel of $\Phi$ corresponding to $b\in \bar\R:=\R\cup\{\pm\infty\}$ and by $$K_c  := \{e \in E:\ \Phi(e) = c,\ {\rm d}\Phi(e) = 0\}$$ the set of the critical points of $\Phi$
in $E$ at the critical level $c \in \R$.

The func\-tio\-nal $\Phi$ satisfies the {\em Palais--Smale condition, briefly $(\rm PS)$, at level $c$} ($c \in \R$),
if any sequence $\{u_j\}_{j\in\NN} \subseteq E$ such that
\begin{equation}\label{cer}
\lim_{j \to +\infty}\Phi(u_j) = c\quad\mbox{and}\quad
\lim_{j \to +\infty}\|{\rm d}\Phi(u_j)\|_{E'} = 0
\end{equation}
converges in $E$, up to subsequences. In general,
if $-\infty\leq a<b\leq +\infty$, $\Phi$ satisfies $(\rm PS)$ in $(a,b)$ if so is at each level $c\in (a,b)$.\par

As usual, the classical Ambrosetti-Rabinowitz condition plays a crucial role in proving that every Palais-Smale sequence is bounded, as well as the so called mountain-pass geometry is satisfied. However, even dealing with different problems than ours, several authors studied different assumptions that still allow to apply min-max procedure in order to assure the existence of critical points. For asymptotically linear problems we show a result that moves along this direction.\par

Let us recall some basic notions of the index theory
for an even functional with symmetry group $\Z_2 = \{{\rm id}, -{\rm id}\}$. Let us set
\[
\begin{split}
\Sigma \ :=\ \{A \subseteq E:\ &A \ \hbox{closed and symmetric w.r.t. the origin,}\\
&\hbox{i.e. $-e \in A$ if $e \in A$}\}\end{split}
\]
and
\[
{\mathcal H} :=\{h\in C(E,E): h \mbox{ odd} \}.
\]
\indent For $A \in \Sigma$, $A\ne\emptyset$, the genus of $A$ is
\[
\gamma(A) \ :=\ \inf\{m \in \N:\  \exists \psi \in C(A,\R^m\setminus\{0\})\ \hbox{s.t.}\
\psi(-e) = - \psi(e), \forall e\in A\}
\]
if such an infimum exists, otherwise $\gamma(A) = +\infty$. Assume $\gamma(\emptyset) = 0$.

The index theory $(\Sigma,{\mathcal H}, \gamma)$ related to $\Z_2$ is also
called  genus (we refer for more details to
 \cite[Section II.5]{s}).\par
  The pseudo--index related to the genus and $S\in\Sigma$ is the triplet
$(S, {\mathcal H}^\ast,\gamma^\ast)$ such that
${\mathcal H}^\ast$ is a group of odd homeomorphisms and $\gamma^\ast: \Sigma\longrightarrow \N\cup\{+\infty\}$ is the map defined by
\[
\gamma^\ast(A) :=\min _{h\in {\mathcal H}^\ast}\gamma(h(A)\cap S),  \;\; \forall\, A\in\Sigma
\]
(cf. \cite{b} for more details).

The proof of our main multiplicity result, Theorem \ref{thm2}, is based on the following abstract result proved in \cite[Theorem 2.9]{bbf}.

\begin{thm}\label{group}
Let $a,b, c_0, c_\infty\in\bar\R$,
$-\infty\leq a<c_0<c_\infty<b\leq +\infty$,
$\Phi$ be an even functional, $(\Sigma,\mathcal H,\gamma)$
the genus theory on $E$, $S\in\Sigma$, $(S, {\mathcal H}^\ast,\gamma^\ast)$ the pseudo-index
theory related to the genus and $S$, with
$$
{\mathcal H}^\ast :=\{h\in{\mathcal H}: h \mbox{ bounded homeomorphism}
$$
$$
\quad\quad\quad\quad\quad\,\,\,\mbox{such that}\,\, h(e)=e \mbox{ if } e\not\in \Phi^{-1}((a,b))\}.
$$
Assume that$:$
\begin{itemize}
\item[$(\rm i)$] the
functional $\Phi$ satisfies $(\rm PS)$ in $(a,b)$$;$
\item[$(\rm ii)$] $S\subseteq \Phi^{-1}([c_0,+\infty))$$;$
\item[$(\rm iii)$] there exist $\tilde k\in\N$ and $\tilde A\in\Sigma$ such that
$\tilde A\subseteq \Phi^{c_\infty}$ and $\gamma^\ast(\tilde A)\geq \tilde k$.
\end{itemize}
Then, setting $\Sigma_i^\ast:= \{A\in \Sigma: \gamma^\ast(A)\geq i\}$, the numbers
\begin{equation}\label{value}
c_i:=\inf_{A\in \Sigma_i^\ast}\sup_{e\in A}\Phi(e), \quad \forall\, i\in \{1,\ldots, \tilde k\},
\end{equation}

\noindent are critical values for $\Phi$ and
\[
c_0\leq c_1\leq \ldots\leq c_{\tilde k}\leq c_\infty.
\]
Furthermore, if $c=c_i=\ldots =c_{i+r}$, with $i\geq 1$ and $i+r\leq \tilde k$, then $\gamma(K_c)\geq r+1$.
\end{thm}

\begin{remark}
If $\Phi\in C^{1}(E, \R)$ and $E$ is a Hilbert space, by Riesz theorem there exists a unique $\Phi'(u)\in E$ such that
$$
\langle \Phi'(u),v \rangle= {\rm d}\Phi(u)[v] \quad \forall v\in E.
$$
$\Phi'(u)$ is called the gradient of $\Phi$ in $u$. With this notation, a critical point of $\Phi$ is a solution to $\Phi'(u)=0$.
\end{remark}

\begin{remark}\label{baba}{\rm
In the applications, a lower bound for the pseudo--index of a suitable $\tilde A$ as in $(\rm iii)$ of Theorem \ref{group} is needed: considering  the genus theory $(\Sigma,\mathcal H,\gamma)$ on $E$
and $V, W$ two closed subspaces of $E$, if
$$\dim V <+\infty \quad \hbox{ and } \quad  \codim\, W<+\infty,$$
then
\[
\gamma(V\cap h(\partial B\cap W))
\geq \dim V - \codim\, W
\]
for every bounded $h\in\mathcal{H}$ and every open bounded symmetric neighbourhood $B$ of $0$ in $E$ (cf.
\cite[Theorem A.2]{bbf}).
}
\end{remark}


\section{periodic solutions via saddle point theorem}\label{principale}

In order to obtain weak solutions of problem (\ref{R}), we study the critical points of the following functional
\begin{equation}\label{1bis}
\J(v):=\displaystyle\frac 12 \|v\|^{2}_{\X}
-\displaystyle\frac{\lambda_\infty \kappa_{s}}2 |\T(v)|_{L^{2}(0, T)^{N}}^2 dx - \kappa_{s}\int_{\partial^{0}\mathcal{S}_{T}} F(x,\T(v)) dx,
\end{equation}
defined on the Hilbert space $\X$ and where, as usual, we set
$$F(x,t):=\displaystyle\int_{0}^{t} f(x,\tau) d\tau.$$

Now, we prove that the functional $\J$ is smooth and it has the geometric structure required by the
Saddle Point Theorem (see, for instance, \cite[Theorem 4.6]{r}).\par
Here, we use the structural assumptions on $f$ to deduce some bounds from above and below for the nonlinear term and
its primitive. This part is quite standard and does not take into account the nonlocal features of the problem.\par
\indent By the growth condition $(f_1)$ and \eqref{conmu2} for all $\varepsilon>0$ there exists $a_\varepsilon>0$ such that
\begin{equation}\label{infymin}
|f(x,t)|\leq\varepsilon |t| + a_\varepsilon,
\quad \hbox{ for a.e. } x\in (0, T)^{N},\,\,  \forall\,  t\in\R.
\end{equation}
so, in particular, we deduce that
\begin{equation}\label{propF}
|F(x,t)|\leq c_1(1+t^2), \quad\hbox{ for a.e. } x\in (0, T)^{N}, \forall\, t\in\R.
\end{equation}
\indent Then, by using Theorem \ref{compacttracethm}, it follows that the functional $\J$ is well-defined and $\J \in C^{1}(\X, \R)$.
In particular, its gradient is given by
\begin{align}\begin{split}\label{differo}
\langle \J'(v), \varphi \rangle  &= \iint_{\mathcal{S}_{T}} y^{1-2s} (\nabla v \nabla \varphi + m^{2} v \varphi) dxdy \\
&-{\lambda_\infty}\kappa_{s} \int_{\partial^{0}\mathcal{S}_{T}} \T(v) \T(\varphi) dx - \kappa_{s}\int_{\partial^{0}\mathcal{S}_{T}} f(x,\T(v)) \T(\varphi) dx,
\end{split}\end{align}
for every $\varphi \in \X$.

Therefore, in order to apply critical point methods, we have
to check the validity of the Palais-Smale condition, that is
\begin{center}
\textit{for any $c\in \RR$ any sequence $\{v_j\}_{j\in\NN}$ in $\X$ such that\\
$\mathcal J(v_j)\to c \,\, \mbox{and}\,\, \sup\Big\{
\big|\langle\,\mathcal J'(v_j),\varphi\,\rangle \big|\,: \;
\varphi\in \X\,,
\|\varphi\|_{\X}=1\Big\}\to 0$\\
as $j\to +\infty$, admits a subsequence strongly convergent in
$\X$\,.}
\end{center}

Hence, in the next proposition we prove that in the non-resonant case the functional $\J$ satisfies the $(\rm PS)$ condition.
\begin{prop}\label{palmsmin}
Assume that $(f_1)$ and \eqref{conmu2} of $(f_2)$ hold. Then,
if $\lambda_\infty\not\in\sigma((-\Delta+m^{2})^{s})$, the functional $\J$ verifies the $(\rm PS)$ in $\R$.
\end{prop}
\begin{proof}
Let  $c\in\R$ and $\{v_{j}\}_{j\in \NN}$ be a sequence in $\mathbb{X}^{s}_{T}$ such that
\begin{equation}\label{psj}
\J(v_{j})\rightarrow c \mbox{ and } \|\J'(v_{j})\|_{\mathbb{X}^{-s}_{T}}\rightarrow 0 \mbox{ as } j\rightarrow \infty,
\end{equation}
where $\mathbb{X}^{-s}_{T}$ is the dual space of $\X$.
Hence, it follows that
\begin{align}\begin{split}\label{ro3min}
&\langle v_j,\varphi\rangle_{\mathbb{X}^{s}_{T}} - \lambda_\infty \kappa_{s}\int_{\partial^{0}\mathcal{S}_{T}}  \T(v_j) \T(\varphi) dx \\
&\,\,\,\,\,\,\,\,\,\,\,\,\,\,\,\,\,\,\,\,\,\,\,- \kappa_{s}\int_{\partial^{0}\mathcal{S}_{T}} f(x, \T(v_j)) \T(\varphi) dx = o(1)
\end{split}\end{align}
for all $\varphi\in \mathbb{X}^{s}_{T}$, where $o(1)$ denotes an infinitesimal sequence.\par
\indent First of all, as usual when using variational methods, we prove the boundedness
of a Palais-Smale sequence for $\J$. Hence, we prove that $\{v_j\}_{j\in \NN}$ is bounded in $\X$.
Let us assume by contradiction that
\begin{equation}\label{roxmin}
\|v_{j}\|_{\mathbb{X}^{s}_{T}}\rightarrow +\infty \quad \hbox{ as } j\rightarrow + \infty.
\end{equation}
\indent By setting $w_j:=\displaystyle\frac {v_j}{\|v_j\|_{\mathbb{X}^{s}_{T}}}$, it is clear that $\{w_j\}_{j\in \NN}$ is bounded in $\mathbb{X}^{s}_{T}$ and by using Theorem \ref{tracethm}, there exists $w\in \mathbb{X}^{s}_{T}$ such that, up to subsequences, it results
\begin{align}\begin{split}\label{ro5min}
&w_j\rightharpoonup w \quad \hbox{ weakly in }  \X  \\
&\T(w_j)\rightarrow \T(w) \quad \hbox{ strongly in } L^{2}(0,T)^{N}.
\end{split}\end{align}
\indent Now, by using (\ref{differo}) with $\varphi:=w_j-w$ and dividing by $\|v_j\|_{\X}$, we have
\begin{align}\begin{split}\label{ro8min}
&\<w_j,w_j-w\>_{\X}=\lambda_\infty \kappa_{s}\int_{\partial^{0}\mathcal{S}_{T}} \T(w_j) (\T(w_j)-\T(w))\, dx\\
&\,\,\,\,\,\,\,\,\,\,\,\,\,\,\,\,\,\,\,\,\,\,\,\,\,\,\,-\kappa_{s}\int_{\partial^{0}\mathcal{S}_{T}} \frac{f(x,\T(v_{j}))}{\|v_j\|_{\X}}(\T(w_j)-\T(w)) \, dx+o(1).
\end{split}\end{align}

\indent Moreover, by \eqref{ro5min} we easily have
\begin{equation}\label{381}
\left| \int_{\partial^{0}\mathcal{S}_{T}} \T(w_j) (\T(w_j)-\T(w)) dx\right|\leq
|\T(w_j)|_2|\T(w_j)-\T(w)|_2=o(1).
\end{equation}
\indent Now, relations \eqref{infymin}, \eqref{roxmin}, and \eqref{ro5min} yield

\begin{eqnarray}\label{382}
 \left| \int_{\partial^{0}\mathcal{S}_{T}} \frac{f(x,\T(v_j))}{\|v_j\|_{\X}}(\T(w_j)-\T(w)) dx \right| &\leq& \varepsilon \frac{|\T(w_j)|_2|\T(w_j)-\T(w)|_2}{\|v_j\|_{\X}}\nonumber\\
               &&+ \frac{a_\varepsilon |\T(w_j)-\T(w)|_1}{\|v_j\|_{\X}}\nonumber \\
&=& o(1)
\end{eqnarray}

\noindent Putting together \eqref{ro8min}, (\ref{381}) and (\ref{382}) we get
$$
\<w_j, w_j-w\>_{\X} =o(1).
$$
\indent Hence, in particular, we also have
\begin{equation}\label{ro11min}
w_j\rightarrow  w \quad \hbox{ strongly in }  {\X},
\end{equation}
and, by the definition of $w_j$, $w\not=0$.

Now, dividing \eqref{ro3min} by $\|v_j\|_{\X}$ we have
\begin{align}\begin{split}\label{ro8xxmin}
&\<w_j,\varphi\>_{\X}
= \lambda_\infty \kappa_{s}\int_{\partial^{0}\mathcal{S}_{T}}  \T(w_j)\,\T(\varphi) dx \,\\
&\,\,\,\,\,\,\,\,\,\,\,\,\,\,\,\,\,\,\,\,\,\,\,\,\,\,\,\,\,\,\,\,\,\,\,\,\,\,\,\,\,\,\,\,\,\,\,-
\kappa_{s}\int_{\partial^{0}\mathcal{S}_{T}} \frac{f(x,\T(v_j))}{\|v_j\|_{\X}}\T(\varphi)dx + o(1).
\end{split}\end{align}

Let us observe that \eqref{infymin}, \eqref{roxmin} and \eqref{ro5min} give
\begin{equation}\label{roxxmin}
\lim_{j\rightarrow + \infty} \int_{\partial^{0}\mathcal{S}_{T}} \frac{f(x,\T(v_j))}{\|v_j\|_{\X}}\T(\varphi)dx =0,\quad \forall\, \varphi\in {\X}.
\end{equation}
\indent
Then, by using \eqref{ro11min} and \eqref{roxxmin}, and passing to the limit in \eqref{ro8xxmin} as $j\rightarrow +\infty$, we get
$$
\<w,\varphi\>_{\X}
= \lambda_\infty \kappa_{s}\int_{\partial^{0}\mathcal{S}_{T}}  \T(w)\,\T(\varphi) dx, \quad \forall\, \varphi\in \X,
$$
that is $u:=\T(w)$ is a weak solution of the following non-local problem
 \begin{equation*}\label{Pauto}
\left\{
\begin{array}{ll}
(-\Delta+m^{2})^{s}u=\lambda u &\mbox{ in } (0,T)^{N}   \\
u(x+Te_{i})=u(x)    &\,\forall x \in \R^{N},\, i\in \ZZ[1,N]
\end{array}
\right.
\end{equation*}
 \noindent which gives a contradiction because of $\lambda_\infty\notin\sigma((-\Delta + m^{2})^{s})$. Thus we have proved that $\{\|v_{j}\|_{\X}\}_{j\in \NN}$ is bounded.\par

Now it remains to check the validity of the Palais-Smale condition, that
is we have to show that every Palais-Smale sequence for $\J$
strongly converges in $\X$, up to a subsequence.\par

Then, up to subsequence, there exists $v\in \X$ such that
\begin{align}\begin{split}\label{11}
&v_j\rightharpoonup v \quad \hbox{ weakly in }  \X \\
&\T(v_j)\rightarrow \T(v) \quad \hbox{ strongly in }  L^2(0, T)^{N}.
\end{split}\end{align}
By \eqref{psj} and \eqref{11} one has
\begin{equation}\label{B}
\langle \J'(v_j), v_j- v\rangle \rightarrow  0, \quad \hbox{ as } j\rightarrow +\infty
\end{equation}
and, by \eqref{infymin}, we also have
\begin{equation}
\int_{\partial^{0}\mathcal{S}_{T}}|f(x,\T(v_j))||\T(v_j)-\T(v)|\, dx\rightarrow 0,  \quad \hbox{ as } j\rightarrow +\infty.
\end{equation}
\indent Hence, by \eqref{differo} and \eqref{B}, it follows that
\begin{equation}\label{F}
\|v_{j}\|_{\X}^{2} - \langle v_{j}, v \rangle_{\X} \rightarrow 0
\end{equation}
as $j\rightarrow +\infty$.
Then, by \eqref{11} and \eqref{F}, one has
$$
v_j\rightarrow v \quad \hbox{ strongly in }  \X.
$$
The proof is now complete.
\end{proof}

\indent
Taking into account that the $(\rm PS)$ compactness condition is satisfied, we can prove our first existence result.
\begin{proof}
[Proof of Theorem \ref{thm1}.]
We aim to apply the Saddle Point Theorem due to Rabinowitz \cite[Theorem 4.6]{r}.
Putting together (\ref{C}), (\ref{propF}), and trace inequality, we have
\begin{align}
\J(v)\geq \frac 12\left(1-\frac{\lambda_\infty}{\lambda_{h+1}} - \frac{c_1}{\lambda_{h+1}}\right)\|v\|^2_{\X}-\kappa_{s} c_1 T^{N},
\end{align}
for every $v\in\mathbb{V}^{\perp}_{h}$.\par
\indent Then we get
\begin{equation}\label{dagiu}
\J(v)\geq c_2, \quad\forall\, v\in\mathbb{V}^{\perp}_{h},
\end{equation}
provided that $h$ large enough.\par
\indent Now, by \eqref{infymin}, we can see that fixed $\varepsilon>0$, we can find $C_\varepsilon>0$ such that
\begin{align}\begin{split}\label{E}
&\J(v)\leq \frac 12\|v\|_{\X}^2 -\kappa_{s}\frac {\lambda_\infty}{2} |\T(v)|_2^2\\
&\,\,\,\,\,\,\,\,\,\,\,\,\,\,\,\,\,\,\,+ \kappa_{s}\frac \varepsilon{2}|\T(v)|^2_2 + \kappa_{s}C_\varepsilon|\T(v)|_2,
\end{split}\end{align}
for every $v\in \X$.\par
\indent Let $\lambda_h<\lambda_\infty$ and take $\varepsilon>0$ such that $\lambda_h+\varepsilon<\lambda_\infty$.\\
Putting together (\ref{E}) and (\ref{eqnorm}), we obtain for any $v \in \mathbb{V}_{h}$
\begin{align}\begin{split}\label{noi}
&\J(v)\leq \frac{\kappa_{s}}{2}\left(\lambda_h+\varepsilon - \lambda_\infty \right)|\T(v)|_{2}^2\\
&\,\,\,\,\,\,\,\,\,\,\,\,\,\,\,\,\,\,\,+\kappa_{s} C_{\varepsilon} |\T(v)|_{2}.
\end{split}\end{align}
\indent By using (\ref{eqnorm}), we can see that for $v\in \mathbb{V}_{h}$, $ |\T(v)|_{2}\rightarrow +\infty$
when $\|v\|_{\X} \rightarrow +\infty$, so (\ref{noi}) implies that $\J(v)\rightarrow -\infty$ as $\|v\|_{\X} \rightarrow +\infty$ and $v\in \mathbb{V}_{h}$.\par
\indent Therefore we can find a positive constant $c_5$ such that
\begin{equation}\label{dasu}
\J(v)\leq -c_3, \quad\forall\, v\in\mathbb{V}_{h}.
\end{equation}
Hence, the thesis follows by Proposition \ref{palmsmin}, (\ref{dagiu}) and (\ref{dasu}).\par
Finally, if $\lambda_\infty< \lambda_1$, we can prove the existence of a weak solution for our problem by using direct minimization techniques.
\end{proof}

\section{proof of Theorem \ref{thm2}}

This section is devoted to establish a multiplicity result for (\ref{P}). Under our conditions on the nonlinear term $f$, one can obtain not only existence critical point theorems, but also
sharper multiplicity results when the functionals are symmetric. In \cite{[BCS2],[BCS], bmb1,bmb2}, multiplicity
results for critical points of even functionals are stated, and their proofs are based on
the use of a pseudo-index theory; see, for instance, \cite{bbf,b} as general references on this topics.\par

Following this approach, we begin proving the following lemmas proving the necessary geometric features of the energy functional $\J$.
\begin{lem}\label{edi}
Assume that conditions $(f_1)$ and $(f_2)$ hold.
Let $\lambda_h$ be as in $(C_{\lambda_{\infty}})$ and $\mathbb{V}^{\perp}_{h-1}$ as defined in \eqref{defvhort}. Then, there exist two positive constants $\rho$ and $c_0$ such that, setting
$$S_\rho:=\{v\in \X: \|v\|_{\X}= \rho\},$$ the functional $\J$ in \eqref{1bis} verifies
\begin{equation}\label{merco}
\J(v)\geq c_0, \quad \forall\, v\in S_\rho\cap \mathbb{V}_{h-1}^{\perp}.
\end{equation}
\end{lem}
\begin{proof}
By using $(f_2)$ we know that
$$
\lim_{|t|\rightarrow +\infty}\frac{F(x,t)}{t^{2}} \ = 0
$$
and
$$
\lim_{t\rightarrow 0}\frac{F(x,t)}{t^{2}} \ = \frac{\lambda_0}2,
$$
uniformly with respect to almost every $x \in (0, T)^{N}$. Moreover, $\lambda_0<0$ in view of condition $(C_{\lambda_{\infty}})$.\par
Then,
for every $\varepsilon>0$ there exist $r_\varepsilon(\geq 1)$ and $\delta_\varepsilon>0$
such that
\begin{equation}
\label{pr1}
|F(x,t)|\leq \frac\varepsilon 2t^2, \quad \quad \hbox{if $|t|> r_\varepsilon$ }
\end{equation}
and
\begin{equation}
\label{pr2}
\left|F(x,t)-\frac {\lambda_0}2t^2\right|\leq \frac\varepsilon 2t^2, \quad \quad \hbox{if $|t|< \delta_\varepsilon$}
\end{equation}\par
\noindent and for almost every $x\in (0, T)^{N}$.\par
On the other hand, by $(f_1)$, taking any constant $$q\in \left[0, \displaystyle\frac{4s}{N-2s}\right),$$  there exists $k_{r_\varepsilon}>0$ such that
\begin{equation}\label{pr3}
|F(x,t)|
\leq k_{r_\varepsilon}{|t|}^{q+2}, \quad \hbox{if $\delta_\varepsilon\leq |t|\leq r_\varepsilon$},
\end{equation}
and for almost every $x\in (0, T)^{N}$.\par
Putting together \eqref{pr1}--\eqref{pr3}, we can deduce that for any $\varepsilon>0$ there exists $k_\varepsilon>0$ such
that
$$
F(x,t)
\leq\frac{\lambda_0 + \varepsilon}2t^2 + {k_\varepsilon}|t|^{q+2},
$$
for almost every $x\in (0, T)^{N}$ and for all $t\in\R$.\par
As a consequence
\begin{align*}
\int_{\partial^{0}\mathcal{S}_{T}} F(x,\T(v)) \; {\rm d}x\leq \frac{\lambda_0 + \varepsilon}2|\T(v)|^2_2 +
{k_\varepsilon}|\T(v)|^{q+2}_{q+2},
\end{align*}
for every $v\in \X$.\par
From this and by using Theorem \ref{compacttracethm}, for a suitable $k'_\varepsilon>0$ we can see that
\begin{equation}\label{nicola}
\J(v)\geq \frac{1}{2}\|v\|_{\X}^2 - \frac{\kappa_{s}(\lambda_{\infty} + \lambda_{0} + \varepsilon)}{2}|\T(v)|_2^2 -
k_\varepsilon' \|v\|_{\X}^{q+2},
\end{equation}
for every $v\in \X$.\par
\indent Then, by \eqref{C} and \eqref{nicola}, it follows that
\begin{align*}
\J(v)\geq \frac{1}{2} \left(1 - \frac{\lambda_\infty + \lambda_0 + \varepsilon}{\lambda_{h}}\right)\|v\|_{\X}^2 - k_\varepsilon' \|v\|_{\X}^{q+2},
\end{align*}
for every $v\in \mathbb{V}_{h-1}^{\perp}$.\par
 Hence, by $(C_{\lambda_{\infty}})$, for a suitable $\varepsilon$, there exists $k''_\varepsilon>0$ such that
\begin{align*}
\J(v)\geq  k''_\varepsilon\|v\|_{\X}^2- k_\varepsilon'\|v\|_{\X}^{q+2},
\end{align*}
for every $v\in \mathbb{V}_{h-1}^{\perp}$.\par
\indent Thus we can find $\rho$ sufficiently small and $c_0>0$ such that inequality
\eqref{merco} holds.
\end{proof}

\begin{lem}\label{simp}
Assume that $(f_1)$ and \eqref{conmu2} hold.
Let $\lambda_k$ as in $(C_{\lambda_{\infty}})$, $\mathbb{V}_{k}$ as in \eqref{defvh} and $c_0$ as in Lemma \ref{edi}. Then, there exists   $c_\infty>c_0$ such that the functional $\J$ in \eqref{1bis} verifies
\begin{equation}\label{abil}
\J(v)\leq c_\infty, \quad \forall\, v\in \mathbb{V}_k.
\end{equation}
\end{lem}
\dimo
By \eqref{E}, taking $\lambda_k$  as in $(C_{\lambda_{\infty}})$ and $\varepsilon>0$ such that $\lambda_k+\varepsilon<\lambda_\infty$, it results that
$$
\J(v)\leq \frac{\kappa_{s}}{2}\left(\lambda_k+\varepsilon - \lambda_\infty \right)|\T(v)|_{2}^2  +\kappa_{s} C_{\varepsilon} |\T(v)|_{2},
$$
for every $v\in \mathbb{V}_k$.\par
Then, in view of (\ref{eqnorm}), there exists $c_\infty=c_\infty(\varepsilon)$ (with $c_\infty>c_0$), such that inequality \eqref{abil} holds.
\hfill$\square$\par
\smallskip
We conclude this section giving the proof of Theorem \ref{thm2}.
\begin{proof}[Proof of Theorem \ref{thm2}]

The idea of the proof consists in applying \cite[Theorem 2.9]{bbf} (recalled in Subsection \ref{INDI}) to the functional $\J$
defined on the Hilbert space $\X$.

By Proposition \ref{palmsmin}, follows that $\J$ is an even functional satisfying $(\rm PS)$ in $\R$.
Let us consider
$\lambda_h$, $\mathbb{V}^{\perp}_{h-1}$, $\rho, c_0$ as in Lemma \ref{edi} and
$\lambda_k, \mathbb{V}_{k}$,  $c_\infty$ as in Lemma \ref{simp}.\par
 Then we consider the pseudo--index theory
$(S_\rho\cap \mathbb{V}_{h-1}^{\perp}, {\mathcal H}^\ast, \gamma^\ast)$ related to the genus,
$S_\rho\cap \mathbb{V}_{h-1}^{\perp}$  and $\J$.\par
In view of Remark \ref{baba},
taking $V:=\mathbb{V}_k$,  $\partial B:= S_\rho$ and $W:=\mathbb{V}_{h-1}^{\perp}$, we get
$$
\gamma\left(\mathbb{V}_k\cap h\left(S_\rho\cap \mathbb{V}^{\perp}_{h-1}\right)\right)\geq \dim \mathbb{V}_k - \mbox{ codim }\mathbb{V}^{\perp}_{h-1}, \quad \forall\, h\in {\mathcal H}^\ast,
$$
which implies
\[
\gamma^\ast(\mathbb{V}_k)\geq k-h+1.
\]
\indent Then we can apply Theorem \ref{group} with $\tilde A:= \mathbb{V}_k$ and $S:=S_\rho\cap \mathbb{V}_{h-1}^{\perp}$ to deduce that $\J$ has at least $k-h+1$ distinct pairs of critical
points corresponding to at most $k-h+1$ distinct
critical values $c_i$, where $c_i$ is as in \eqref{value}.\par
Then, if $Z$ denotes the set of $k-h+1$ distinct pairs of critical
points of $\J$ obtained applying Theorem \ref{group}, the set $\T(Z)$ contains $k-h+1$ distinct pairs of weak solutions of problem \eqref{P}. The proof is complete.
\end{proof}

\begin{remark}\rm{
We point out that, combining the proof of our main results and those of \cite[Theorem 3.1]{[BCS]}, Theorem \ref{thm2} holds if we require that
\[
\lambda_\infty<\ \lambda_h\leq \ \lambda_k\ < \lambda_0 + \lambda_\infty,
\]
instead of condition $(C_{\lambda_{\infty}})$.\par
\noindent As a matter
of fact, we plan to consider further applications of our abstract framework for fractional
equations involving a suitable resonant term in a forthcoming paper.}
\end{remark}

\noindent {\bf Acknowledgements.} The paper has been carried out under the auspices of the INdAM - GNAMPA Project 2016 titled: {\it Problemi variazionali su variet\`a Riemanniane e gruppi di Carnot}. This manuscript was revised during the visit of G.M.B. to the Mathematics Section of the Abdus Salam
International Centre for Theoretical Physics (ICTP) - Trieste, in August 2016. He would like to thank Prof. F.
Rodr\'{i}guez-Villegas for his kind invitation and warm hospitality during the visit at the ICTP. He also express his gratitude to Prof. S. Ouaro and the Research Group from Burkina Faso for the joint scientific activities in Trieste. A special thank goes to Ms. K. Mabilo and Prof. F. Maggi for their strong human support and generosity.

\end{document}